\documentclass[12pt]{amsart}

\author{Cristina Acciarri}

\address{(Cristina Acciarri) via Francesco Crispi, n.81 San Benedetto del Tronto (AP) Italy}
\email{acciarricristina@yahoo.it}

\author{Pavel Shumyatsky} 

\address{(Pavel Shumyatsky) Department of Mathematics, University of Brasilia,
Brasilia-DF, 70910-900 Brazil}

\email{pavel@unb.br}

\thanks{The first author was supported by the Spanish Government, grant
MTM2008-06680-C02-02, partly with FEDER funds. She also thanks the Department of Mathematics of the University of Brasilia where this research was partly conducted. The second author was supported by CAPES and CNPq-Brazil}

\usepackage{amssymb}
\numberwithin{equation}{section}
\usepackage[colorlinks=false]{hyperref}


\title{Fixed points of coprime operator groups}

\newtheorem{theorem}{\sc Theorem}[section]
\newtheorem{lemma}[theorem]{\sc Lemma}
\newtheorem{proposition}[theorem]{\sc Proposition}
\newtheorem{corollary}[theorem]{\sc Corollary}

\newtheorem{conjecture}[theorem]{\sc Conjecture}

\newcommand{\F}{\mathbb{F}}

\theoremstyle{definition}
\newtheorem{definition}[theorem]{Definition}

\begin{document}

\begin{abstract} Let  $m$ be a positive integer and $A$  an elementary abelian group of order $q^r$ with $r\geq 2$ acting on a finite $q'$-group $G$.  We show that if for some integer $d$ such that $2^{d}\leq r-1$ the $d$th derived group of $C_{G}(a)$  has exponent dividing $m$ for any $a \in  A^{\#}$, then $G^{(d)}$ has $\{m,q,r\}$-bounded exponent and if $\gamma_{r-1}(C_G(a))$ has exponent dividing $m$ for any $a\in A^\#$, then $\gamma_{r-1}(G)$ has $\{m,q,r\}$-bounded exponent.
\end{abstract}

\maketitle

   
\section{Introduction}
\label{introduction}
Let $A$ be a  finite group acting coprimely on a finite
group $G$. It is well known that the structure of
the centralizer $C_G(A)$ (the fixed-point subgroup) of
$A$ has strong influence over the structure of $G$.
To exemplify this we mention the following results.

The celebrated theorem of Thompson \cite{T} says that
if $A$ is of prime order and $C_G(A)=1$, then $G$ is
nilpotent. On the other hand, any nilpotent
group admitting a fixed-point-free automorphism of prime
order $q$ has nilpotency class bounded by some
function $h(q)$ depending on $q$ alone. This result
is due to Higman \cite{Higman}. The reader can find in
\cite{Khu1} and \cite{Khu2} an account on the more recent
developments related to these results.
The next result is a consequence of the classification
of finite simple groups \cite{Wang}: If $A$ is a group
of automorphisms of $G$ whose order is coprime to that
of $G$ and $C_G(A)$ is nilpotent or has odd order, 
then $G$ is soluble. Once the group $G$ is known to be soluble, there is
a wealth of results bounding the Fitting height of $G$ in terms of the order of $A$ and the Fitting height of $C_G(A)$. This direction of research was started by Thompson in \cite{thompson2}. The proofs mostly use representation theory in the spirit of the Hall-Higman work \cite{HH}. A general discussion of these methods and their use in numerous fixed-point theorems can be found in Turull \cite{Tu}.

Following the solution of the restricted Burnside problem it was discovered that the exponent of $C_G(A)$ may have strong impact over the exponent of $G$. Remind that a group $G$ is said to have exponent $m$ if $x^m=1$ for every $x\in G$ and $m$ is the minimal positive integer with this property. The next theorem was obtained in \cite{KS}.

\begin{theorem}
\label{q2}
Let $q$ be a prime, m a positive integer and $A$ an elementary abelian group of order $q^{2}$. Suppose that $A$ acts as a coprime group of automorphisms on a finite group $G$ and assume that $C_{G}(a)$ has exponent  dividing $m$ for each $a\in A^{\#}$. Then the exponent of $G$ is $\{m,q\}$-bounded.
\end{theorem}  
Here and throughout the paper $A^{\#}$ denotes the set of nontrivial elements of $A$. The proof of  the above result involves a number of deep ideas. In
particular, Zelmanov's techniques that led to the solution of the restricted Burnside problem \cite{Z1} are combined with the Lubotzky--Mann theory of powerful $p$-groups \cite{LM}, Lazard's criterion for a pro-$p$~group to be $p$-adic analytic \cite{L}, and a theorem of Bakhturin and Zaicev on Lie algebras admitting a group of automorphisms whose fixed-point subalgebra is PI \cite{BZ}.

Another quantitative result of similar nature  was proved in the paper of Guralnick and the second author \cite{GS}.

\begin{theorem}
\label{q3}
Let $q$ be a prime, $m$  a positive integer. Let $G$ be a finite $q'$-group acted on by an elementary abelian group $A$ of order $q^{3}$. Assume that $C_{G}(a)$ has derived group of exponent dividing $m$ for each $a\in A^{\#}$. Then the exponent of $G'$ is $\{m,q\}$-bounded. 
\end{theorem}

Note that the assumption that $|A|=q^{3}$ is essential here and the theorem fails if $|A|=q^{2}$. The proof of Theorem \ref{q3} depends on the classification of finite simple groups.

It was natural to expect that Theorems \ref{q2} and \ref{q3} admit a common generalization that would show that both theorems are part of a more general phenomenon. Let us denote by $\gamma_i(H)$ the $i$th term of the lower central series of a group $H$ and by $H^{(i)}$ the $i$th term of the derived series of $H$. The following conjecture was made in \cite{drei}.

\begin{conjecture}
\label{255} 
Let $q$ be a prime,  $m$ a positive integer and  $A$ an elementary
abelian group of order $q^r$ with $r\ge 2$ acting 
on a finite $q'$-group $G$.\begin{enumerate}
\item If $\gamma_{r-1}(C_G(a))$ has exponent dividing
$m$ for any $a\in A^\#$, then $\gamma_{r-1}(G)$
has $\{m,q,r\}$-bounded exponent;
\item If, for some integer $d$ such that $2^d\le r-1$,
the $d$th derived group of $C_G(a)$ has exponent dividing
$m$ for any $a\in A^\#$, then the $d$th derived
group $G^{(d)}$ has $\{m,q,r\}$-bounded exponent.
\end{enumerate}
\end{conjecture}

The main purpose of the present paper is to confirm Conjecture \ref{255}. Theorem \ref{PR} and Theorem \ref{gamma_PR} show that both parts of the conjecture are correct. The main novelty of the paper is the introduction of the concept of $A$-special subgroups of $G$ (see Section 3). Using the classification of finite simple groups it is shown in Section 4 that the $A$-invariant Sylow $p$-subgroups of $G^{(d)}$ are generated by their intersections with $A$-special subgroups of degree $d$. This enables us to reduce the proof of Conjecture \ref{255} to the case where $G$ is a $p$-group, which can be treated via Lie methods. The idea of this kind of reduction has been anticipated already in \cite{GS}. In Section 6 we give a detailed proof of part (2) of Conjecture \ref{255}. In Section 7 we briefly describe how the developed techniques can be used to prove part (1) of Conjecture \ref{255}.

Throughout the article we use the term ``$\{a,b,c,\dots\}$-bounded" to mean ``bounded from above
by some function depending only on the parameters $a,b,c,\dots$".

\section{Preliminary Results}
\label{preliminary result}

We start with the following elementary lemma.
\begin{lemma}
\label{generation_and_product}
Suppose that a nilpotent group $G$ is generated by subgroups $G_{1},\ldots,G_{t}$ such that 
$\gamma_{i}(G)=\langle \gamma_{i}(G)\cap G_{j} \mid 1 \leq j \leq t \rangle$ for all $i\geq 1$. Then $G=G_{1}G_{2}\cdots G_{t}$.   
\end{lemma}

\begin{proof}
We argue by induction on the nilpotency class $c$ of $G$. If $c=1$, then $G$ is abelian and the result is clear. Assume that $c\geq 2$. Let $K=\gamma_{c}(G)$. Since $K$ is central, it is abelian and we have 
$K=K_{1}K_{2}\cdots K_{t}$,
where 
$K_{j}=K\cap G_{j}$ for $j=1,\ldots,t$. 
By induction we have
 $$
 G=G_{1}G_{2}\cdots G_{t}K=G_{1}G_{2}\cdots G_{t}K_{1}K_{2}\cdots K_{t}.
 $$ 
Since each subgroup  $K_{j}$ is central in $G$ and $K_j\leq G_{j}$, it follows that $G=G_{1}G_{2}\cdots G_{t}$, as required.  
\end{proof}

We now collect some facts  about coprime automorphisms of finite groups.
The two following   lemmas  are well known (see  \cite[5.3.16, 6.2.2, 6.2.4]{GO}).

\begin{lemma}
\label{FG1} 
Let $A$ be a  group of automorphisms of the finite group $G$ with $(|A|,|G|)=1$. 
\begin{enumerate}
\item If $N$ is an  $A$-invariant normal  subgroup of $G$, then \\$C_{G/N}(A)=C_G(A)N/N$;
\item If $H$ is any $A$-invariant $p$-subgroup of $G$, then $H$ is contained in an $A$-invariant Sylow $p$-subgroup of $G$.
\end{enumerate}
\end{lemma} 
 
\begin{lemma}
\label{FG2} 
Let $q$ be a prime, $G$ a finite $q'$-group acted on by an elementary abelian $q$-group $A$ of rank at least  $2$.\  Let $A_1, \dots,A_s$ be the maximal subgroups of $A$.\ If $H$ is an $A$-invariant subgroup of $G$ we have 
$H=\langle C_H(A_1),\dots,C_H(A_s)\rangle$. Furthermore if $H$ is nilpotent then $H=\prod_{i} C_{H}(A_{i})$.
\end{lemma}



We also need the following result, which is a well-known corollary of the classification of finite simple groups.
\begin{lemma}
\label{simple}
Let $G$ be a finite simple group and $A$ a group of automorphisms of $G$ with $(|A|,|G|)=1.$ Then $A$ is cyclic.
\end{lemma}

We conclude this section by citing an important theorem  due to Gasch\"{u}tz. The proof can be found in \cite[p.\ 121]{Huppert} or  in \cite[p.\ 191]{Ro}.
\begin{theorem}
 \label{gaschutz thm}
 Let $N$ be a normal abelian $p$-subgroup of  a finite group $G$ and let $P$ be a Sylow $p$-subgroup of $G$. Then $N$ has a complement in $G$ if and only if $N$ has a complement in $P$. 
 \end{theorem}

\section{$A$-special subgroups}
\label{Aspecial}

In this section we introduce the concept of \emph{$A$-special subgroups} of $G$. For every integer $k\geq 0$ we define $A$-special subgroups of $G$ of degree $k$ in the following way. 

\begin{definition}
\label{DAspecial} 
Let $q$ be a prime and $A$ an elementary abelian $q$-group acting on  a finite $q'$-group $G$.
Let $A_{1},\ldots,A_{s}$ be the subgroups of index $q$ in $A$ and $H$ a subgroup of $G$. 
\begin{itemize}
\item We say that $H$ is an $A$-special subgroup of $G$ of degree $0$ if and only if $H=C_{G}(A_{i})$ for suitable $i\leq s$.  

\item Suppose that $k\geq 1$ and the $A$-special subgroups of $G$ of degree $k-1$ are defined. Then $H$ is an $A$-special subgroup of $G$ of degree $k$ if and only if there exist $A$-special subgroups $J_{1},J_{2}$ of $G$ of degree $k-1$ such that  $H=[J_{1},J_{2}]\cap C_{G}(A_{j})$ for suitable $j\leq s$.  
\end{itemize}
 \end{definition}
Here as usual $[J_{1},J_{2}]$ denotes the subgroup generated by all commutators $[x,y]$ where $x\in J_1$ and $y\in J_2$.
We note that all $A$-special subgroups of $G$ of any degree are $A$-invariant. Assume that $A$ has order $q^r$. It is clear that for a given integer $k$ the number of $A$-special subgroups of $G$ of degree $k$ is $\{q,r,k\}$-bounded. Let us denote this number by $s_{k}$.  

The $A$-special subgroups have certain properties that will be crucial for the proof of the main result of this paper.

\begin{proposition}
\label{PAspecial}
Let $A$ be an elementary abelian $q$-group of order $q^{r}$  with $r\geq 2$ acting on  a finite $q'$-group $G$ and let $A_{1},\ldots,A_{s}$ be the maximal subgroups of $A$. Let $k\geq 0$ be an integer. 

\begin{enumerate}

\item If $k\geq 1$, then every $A$-special subgroup of $G$ of degree $k$ is contained in some $A$-special subgroup of $G$ of degree $k-1$.

\item Let $K$ be an $A$-invariant subgroup of $G$ and let $K_{1},\ldots,K_{t}$ be all the subgroups of the form $K\cap H$, where $H$ is some $A$-special subgroup of $G$ of degree $k$. Let $L_{1},\ldots,L_{u}$ be all the subgroups of the form $K'\cap J$ where $J$ is some $A$-special subgroup of $G$ of degree $k+1$. If $K=\langle K_{1},\ldots,K_{t} \rangle$, then $K'=\langle L_{1},\ldots,L_{u} \rangle$. 

\item Let $R_{k}$ be the subgroup generated by all $A$-special subgroups of $G$ of degree $k$. Then $R_{k}=G^{(k)}$.

\item If $2^{k}\leq r-1$ and $H$ is an $A$-special subgroup of $G$ of degree $k$, then $H$ is contained in the $k$th derived group of $C_{G}(B)$ for some subgroup  $B\leq A$ such that $|A/B|\leq q^{2^{k}}$.

\item Suppose that $G=G'$ and let $N$ be an $A$-invariant subgroup such that $N=[N,G]$. Then for every $k\geq 0$ the subgroup $N$ is generated by subgroups of the form  $N\cap H$, where $H$ is some $A$-special subgroup of $G$ of degree $k$.

\item Let $H$ be an $A$-special subgroup of $G$. If $N$ is an $A$-invariant normal subgroup of $G$, then the image of $H$ in $G/N$ is an $A$-special subgroup of $G/N$. 
 \end{enumerate} 
\end{proposition}
  
\begin{proof}
(1)  If $k=1$ and  $H$ is an $A$-special subgroup of $G$ of degree $1$, then $H=[J_{1},J_{2}]\cap C_{G}(A_{j})$ for a suitable $j\leq s$.  Observe that  $H\leq C_{G}(A_{j})$ and the centralizer $C_{G}(A_{j})$ is an $A$-special subgroup of $G$ of degree $0$. Assume that $k\geq 2$ and use induction on $k$. Let $H$ be an $A$-special subgroup of degree $k$. We know that there exist $A$-special subgroups $J_{1}, J_{2}$ of $G$ of degree $k-1$ such that $H=[J_{1},J_{2}]\cap C_{G}(A_{j})$ for suitable $j\leq s$. By induction $J_{i}$ is contained in some $A$-special subgroup $L_{i}$ of $G$ of degree $k-2$. Observe that $[L_{1},L_{2}]\cap C_{G}(A_{j})$ is an $A$-special subgroup of $G$ of degree $k-1$ and $H\leq [L_{1},L_{2}]\cap C_{G}(A_{j})$, so the result follows.

(2) Set $M=\langle [K_{i},K_{j}] \mid  1\leq i,j \leq t\rangle$. It is clear that each of the subgroups $[K_{i},K_{j}]$ is $A$-invariant. Thus, by Lemma \ref{FG2} each subgroup $[K_{i},K_{j}]$ is generated by  subgroups of the form $[K_{i},K_{j}]\cap C_{G}(A_{l})$, where $l=1,\ldots,s$.  Note that each  subgroup $[K_{i},K_{j}]\cap C_{G}(A_{l})$ is  contained in an $A$-special subgroup of $G$ of degree $k+1$. Hence $M$ is generated by subgroups of the form $M\cap D$, where $D$ ranges through the set of all $A$-special subgroups of $G$ of degree $k+1$.  If $M^*=M\cap D^*$ is such a subgroup we claim that $[M^*,K_{j}]\leq M$ for every $1\leq j\leq t$. Indeed, by (1) we know that there exists  some $A$-special subgroup $H$ of $G$ of degree $k$ such that $D^*\leq H$. This implies that $M^*$ is contained in some $K_{l}$ and so we have  $[M^*,K_{j}]\leq[K_{l},K_{j}]\leq M$, as desired.  Therefore $M$ is normal in $K$ and we conclude that $M=K'$.  The result now follows.  

(3) If $k=0$ the result is immediate from Lemma \ref{FG2}. Therefore we assume that $k\geq 1$ and set $N=R_{k-1}$. By induction on $k$ we assume that $N=G^{(k-1)}$. Let $D_{1},D_{2},\ldots,D_{s_{k-1}}$ be the $A$-special subgroups of $G$ of degree $k-1$ and $H_{1},H_{2},\ldots,H_{s_{k}}$ be the $A$-special subgroups of $G$ of degree $k$. It follows from (2) that $G^{(k)}=\langle [D_{i},D_{j}] \mid  1\leq i,j \leq s_{k-1}\rangle$. Since each subgroup $[D_{i},D_{j}]$ is $A$-invariant, it follows from Lemma \ref{FG2} that it is generated by  subgroups of the form $[D_{i},D_{j}]\cap C_{G}(A_{l})$, where $l=1,\ldots,s$. These are precisely $A$-special subgroups of $G$ of degree $k$ so the result follows.

(4) If $k=0$ this is clear because $H=C_{G}(A_{i})$ for a suitable $i\leq s$ and $|A/A_{i}|=q$. 
Assume that $k\geq 1$ and use induction on $k$. We have $H=[J_{1},J_{2}]\cap C_{G}(A_{j})$ for a suitable $j\leq s $ and $A$-special subgroups $J_{1},J_{2}$ of $G$ of degree $k-1$. By induction there exist subgroups $B_{1},B_{2}\leq A$ such that $|A/B_{i}|\leq q^{2^{k-1}}$ and $J_{i}\leq C_{G}(B_{i})^{(k-1)}$ where $i=1,2$. Set $B=B_{1}\cap B_{2}$.  Observe that 
$H\leq [J_{1},J_{2}]\leq [C_{G}(B_{1})^{(k-1)},C_{G}(B_{2})^{(k-1)}]\leq [C_{G}(B)^{(k-1)}, C_{G}(B)^{(k-1)}]$. Thus $H\leq C_{G}(B)^{(k)}$ and $|A/B|\leq q^{2^{k}}$, as required.

(5) For $k=0$ this follows from Lemma \ref{FG2}. Assume that $k\geq 1$ and use induction on $k$. 
Let $N_{1},\ldots, N_{t}$ be all the subgroups of the form $N\cap H$ where $H$ is some $A$-special subgroup of degree $k$ and set  $M=\langle N_{1},\ldots, N_{t}\rangle$. We want to show that $N=M$. 

Since $G=G'$  by (3) $G$ can be generated by all $A$-special subgroups of degree $k$, for any $k\geq1$. Thus $G=\langle H_{1},\ldots, H_{s_{k}}\rangle$, where $H_{j}$ is $A$-special subgroup of degree $k$. Lemma \ref{FG2} shows that for all $i$ and $j$ the commutator $[N_{i},H_{j}]$ is generated by subgroups of the form $[N_{i},H_{j}]\cap C_{G}(A_{l})$, where $l=1,\dots,s$. Note that each subgroup $[N_{i},H_{j}]\cap C_{G}(A_{l})$ is contained in $N$ since $N=[N,G]$ and on the other hand  it is also contained in some $A$-special subgroup of degree $k$, so $[N_{i},H_{j}]\cap C_{G}(A_{l})\leq N_{m}$ for a suitable $m\leq t$.  This implies that  $M$ is normal in $G$. 

Let  now $L_{1},\ldots, L_{u}$ be all the subgroups of the form $N\cap K$ where $K$ is some $A$-special subgroup of degree $k-1$,  so  by induction  we can assume that $N=\langle L_{1},\ldots,L_{u} \rangle$.  For all $i$ and $j$, using the argument as above, it is easy to show that  $[L_{i},H_{j}]$ is generated by subgroups of the form  $[L_{i},H_{j}]\cap C_{G}(A_{l})$, where $l=1,\dots,s$ and that  each such a subgroup is contained in some $N_{m}$, for a suitable $m\leq t$.

If $M=1$, then,  for all $m\leq t$, $N_{m}=1$ and so $[L_{i},H_{j}]=1$ for all $i$ and $j$.  Hence $N$ is central in $G$ but this is a contradiction because $N=[N,G]$.   Assume  now that $M$ is a nontrivial subgroup  strictly contained  in $N$. Since  we have shown that $M$ is normal we can pass to the quotient $G/M$. In the quotient $N/M$ is central so $[N,G]\leq M$ but this contradicts the assumption that $M<N$. Thus we conclude that $M$ must be equal to $N$.     

(6) This is immediate from  Lemma \ref{FG1}(1) and the definitions. 
\end{proof}

\section{Some generation results}
\label{generation_results}

Throughout this section let $q$ be a prime, $G$ a finite $q'$-group and $A$ an elementary abelian group of order $q^{r}$ acting on $G$.   
We will show that  if $P$ is an $A$-invariant Sylow  $p$-subgroup of $G^{(d)}$, then it can be generated by its intersections with $A$-special subgroups of $G$ of degree $d$.  

\begin{theorem}
\label{generation1}
Assume $r\geq 2$. Let $P$  be an $A$-invariant Sylow $p$-subgroup of $G^{(d)}$ for some fixed integer $d\geq 0$. Let $P_{1},\ldots,P_{t}$ be the subgroups of the form $P\cap H$ where $H$ is some $A$-special subgroup of $G$ of degree $d$. Then $P=\langle P_{1},\dots,P_{t}\rangle$.
\end{theorem} 

We first handle the case where $G$ is a direct product of simple groups.
 \begin{lemma}
 \label{product_of_simple}
 Assume  that $r\geq 2$ and $G$ is a direct product of nonabelian simple groups. Let $P$ be an $A$-invariant Sylow $p$-subgroup of $G$, and  for some fixed integer $d\geq 0$ let $P_{1},\ldots,P_{t}$ be all the subgroups of the form $P\cap H$, where $H$ is some $A$-special subgroup of $G$ of degree $d$. Then $P=\langle P_{1},\dots,P_{t}\rangle$. 
 \end{lemma}
 
 \begin{proof} Let $G=S_{1}\times \cdots \times S_{m}$. By induction on the order of $G$  we may assume that $A$ permutes transitively the simple factors $S_{1},\ldots,S_{m}$. 

We will now use induction on $r$ to show that without loss of generality it can be assumed that $A$ acts on $G$ faithfully. Suppose that some element $a\in A^{\#}$ acts on $G$ trivially. Thus, $C_{G}(a)=G$.

Since $G$  is a product of nonabelian simple groups  it follows that  $[C_{G}(a),C_{G}(a)]=G$ and $[C_{G}(a),C_{G}(a)]\cap C_{G}(a)=G$. Thus if $r=2$, then $G$ itself is an $A$-special subgroup of degree $1$. From this it is easy to see that $G$ is an $A$-special subgroup of degree $d$ and the lemma follows immediately. 

Suppose now that $r\geq 3$ and let $A_{1},\ldots,A_{s}$ be the maximal subgroups of $A$.  Put $\overline{A}=A/\langle a\rangle$. If $A_{1},\ldots A_{t}$ are the maximal subgroups of $A$ containing $a$, then $\overline{A_{1}},\ldots, \overline{A_{t}}$ are maximal subgroups of $\overline{A}$. Then  $C_{G}(\overline{A_{i}})=C_{G}(A_{i})$ for all $i \leq s$, and so we can consider $\overline{A}$ instead of $A$ and use induction on $r$. 

Thus, from now on we assume that $A$ is faithful on $G$. Let $B$ be the stabilizer of $S_{1}$ in $A$.  Then by Lemma \ref{simple} $B$ is cyclic. Remark that if $b\notin B $, then $C_{G}(b)$ is a product of simple groups. Indeed if we consider all the $b$-orbits, then it is not difficult to see that $C_{G}(b)$ is the product of  the diagonal subgroups of these $b$-orbits, i.e., $C_{G}(b)$ is a product of simple groups, one for each $b$-orbit.

Suppose  that $r=2$ and $B\neq 1$.  Let  $a$ be a nontrivial element of $B$ and choose $b\in  A$ that permutes  $S_{1},\ldots,S_{q}$. Observe that  the case  where $A=B$ does not happen because of Lemma \ref{simple}. Since $b\notin B$ from the above remark we know that $C_{G}(b)=diag(S_{1}\times \cdots \times S_{q})$ is a diagonal subgroup of $G$. On the other hand it follows from the Thompson Theorem \cite{T} that  $C_{S_{1}}(a) \neq 1$ and this  holds also  for the other factors $S_2,\dots,S_q$ because $a$ normalizes each of the simple factors. Thus $C_{G}(a)=C_{S_{1}}(a)\times \cdots \times C_{S_{q}}(a)$ and we have 
\begin{equation}
\label{eq1}
\begin{split}
[C_{G}(a),C_{G}(b)]= &\\
[C_{S_{1}}(a),diag(S_{1}\times &\cdots \times S_{q})]\times\cdots \times [C_{S_{q}}(a),diag(S_{1}\times \cdots \times S_{q})].
  \end{split}\end{equation}
Furthermore observe that  for  any $j=1,\ldots,q$
\begin{equation} 
\label{eq2}
[C_{S_{j}}(a),diag(S_{1}\times \cdots \times S_{q})]=[C_{S_{j}}(a),S_{j}]=S_{j},
\end{equation}
 where the first equality follows from the fact that the simple  factors commute each other and the second one holds since $[C_{S_{j}}(a),S_{j}]$ is a nontrivial normal subgroup of $S_{j}$.  By (\ref{eq1}) and (\ref{eq2})  we see that $[C_{G}(a),C_{G}(b)]=S_{1}\times \cdots \times S_{q}=G.$  
Thus, for any $c\in A^{\#}$, $C_{G}(c)=[C_{G}(a),C_{G}(b)]\cap C_{G}(c)$ and so the centralizer $C_{G}(c)$ is also an $A$-special subgroup of degree $1$.  We deduce  that, for  any $a\in A^{\#}$,  the centralizer  $C_{G}(a)$ is an $A$-special subgroup of $G$ of any degree.  Since Lemma \ref{FG2} tells us that $P$ can be generated by  subgroups of the form $P\cap C_{G}(A_{i})$ where $A_{i}$ are the maximal subgroups of $A$, the result follows. 

Next, assume that  $r=2$ and $B=1$.  Note that $A$ permutes the factors $S_{1},\ldots,S_{q^{2}}$ and,  for any $a\in A^{\#}$,  the centralizer $C_{G}(a)$ is a product of $q$ simple groups, one for each $a$-orbit. Thus $C_{G}(a)$ is perfect and, in particular, it is an $A$-special subgroup of any degree for all $a\in A^{\#}$.  The lemma follows.   

Finally assume that $r\geq3$. Since $B$ is cyclic we have  $|A:B|\geq q^{2}$.  Hence we can choose a subgroup $E$ of type $(q,q)$ that intersects $B$ trivially.  Note that  for  all $a\in E^{\#}$, the centralizer $C_{G}(a)$ is a product of simple groups. Moreover Lemma \ref{FG2} shows that $P=\prod_{a\in E^{\#}}C_{P}(a)$. Therefore it is sufficient to prove that for each $a\in E^{\#}$ the subgroup $C_{G}(a)\cap P$ is generated by its intersections with all the $A$-special subgroups of $G$ of degree $d$. 
 
Fix $a$ in $E^{\#}$. Let $D$ be the group of automorphisms induced on $C_{G}(a)$ by $A$. By induction $C_{G}(a)\cap P$ is generated by subgroups of the form $(C_{G}(a)\cap P)\cap H$, where $H$ ranges through  the set of $D$-special subgroups of $C_{G}(a)$ of degree $d$.  We now remark that any $D$-special subgroup of $C_{G}(a)$ of  any degree is in fact  an $A$-special subgroup of $G$ of the same degree.  This follows form  Definition \ref{DAspecial} and from the fact that if $A_{i}$ is a maximal subgroup of $A$ containing $a$, then there exists a maximal subgroup $D_{j}$ of $D$ such that $C_{C_{G}(a)}(D_{j})=C_{G}(A_{i})$.  Thus we can conclude that $C_{G}(a)\cap P$ is generated by subgroups of the form $(C_{G}(a)\cap P)\cap H$, where now $H$ can be regarded as an $A$-special subgroup of $G$ of degree $d$. The proof is now complete.
 \end{proof}

 We now are ready to complete the proof of Theorem  \ref{generation1}.
 
 \begin{proof}[Proof of Theorem \ref{generation1}]
Let $G$ be a counterexample of minimal order and let $N$ be a minimal  normal $A$-invariant subgroup of  $G$. Set  $X=\langle P_{1},\ldots,P_{t}\rangle$. By minimality  and Proposition \ref{PAspecial}(6) $PN=XN$. To prove that $P=X$ it is sufficient to show that $P\cap N \leq X$. 

First suppose that $N$ is a $p'$-group. In this case the intersection $P\cap N$ is trivial  and  there is nothing to prove.
 
Next suppose that $N$ is perfect. Since $N$ is characteristically simple,  $N$ is a product of nonabelian simple groups. It follows from Lemma \ref{product_of_simple} that $P\cap N$ is contained in $X$ and we are done.

Thus, it remains  to consider the case where $N$ is a $p$-group. Suppose that $G \neq G'$. By induction we know that  every $A$-invariant Sylow $p$-subgroup of $G^{(d+1)}$ is generated  by its intersections with  all the $A$-special subgroups of $G'$ of degree $d$ . Therefore we can pass to the quotient $G/G^{(d+1)}$ and assume that  $G^{(d+1)}=1$. This implies that $G^{(d)}$ is abelian and so  we may assume  that $G^{(d)}$ is a $p$-group. Then $G^{(d)}=P$.  It follows from Proposition \ref{PAspecial}(3)  that $P$ is generated by $A$-special subgroups of $G$ of degree $d$  and  the result holds.    

We are reduced to the  case that  $G=G'$. Since $N$ is minimal, either  $N=[N,G]$ or $N\leq Z(G)$.

If $N=[N,G]$ we note that $P\cap N=N$ because $N$ is contained in $P$.   Since $N=[N,G]$, Proposition \ref{PAspecial}(5) shows that  $N$ is generated by  its intersections with all the $A$-special subgroups of $G$ of degree $d$ and  so $N\leq X$, as desired.

Now suppose $N$ central in $G$. Then $N$ is of order $p$ and either $N$ is contained in every maximal subgroup of $P$ or there exists a maximal subgroup $S$ in $P$ such that $P=NS$. 

In the former case $N\leqslant \Phi(P)$. Since we know that $P=XN$, it follows that $P=X$, as required. 

In the latter case, by Theorem \ref{gaschutz thm}, $N$ is also complemented in $G$ and so $G=NH$ for some subgroup $H\leq G$. Since $N$ is central, we have $G=N\times H$. This yields a contradiction because we have  assumed that $G=G'$.  
\end{proof}

We note some consequences of Theorem \ref{generation1}. These facts will be useful later on. 

\begin{lemma}
\label{generation2}
Under the  hypothesis  of Theorem \ref{generation1} let $P^{(l)}$ be the $l$th derived group of $P$. Then $P^{(l)}= \langle P^{(l)} \cap P_{j} \mid 1 \leq j \leq t \rangle$.
\end{lemma}
\begin{proof}
First we want to establish the following fact: 

\begin{equation}
\label{generation}
\begin{aligned}
&\text{The group}\,   P^{(l)} \, \text{is generated by  all the subgroups of the form}\\
&P^{(l)}\cap D,\, \text{where}\,  D \, \text{ranges through the set of all}\, A\text{-special sub-}\\
&\text{groups of} \, G \, \text{of degree}\, d+l.\\ 	
\end{aligned}
\end{equation}
Indeed if  $l=0$ then (\ref{generation}) is exactly  Theorem \ref{generation1}. Assume that $l\geq 1$ and use induction on $l$.  Let $L_{1},\ldots, L_{u}$ be all the subgroups of the form $P^{(l)}\cap J$, where $J$ is some $A$-special subgroup of $G$ of degree $d+l$.  By induction $P^{(l-1)}$ is generated by subgroups of the form $P^{(l-1)}\cap D$, where $D$ is some $A$-special subgroup of $G$ of degree $d+(l-1)$. It now follows from Proposition \ref{PAspecial}(2) that $P^{(l)}=\langle L_{1},\ldots, L_{u}\rangle$ and this concludes the proof of (\ref{generation}). 

Now for $l=0$ the lemma is obvious since by Theorem \ref{generation1}   $P=\langle P_{1},\ldots, P_{t}\rangle$, where  each subgroup $P_{j}$ is of the form $P\cap H$ for some $A$-special subgroup $H$ of $G$ of degree $d$.   
Assume that $l\geq 1$. Proposition \ref{PAspecial}(1) tells us that every $A$-special subgroup  $D$ of degree $d+l$ is contained in some $A$-special subgroup $H$ of degree $d$. Combined with (\ref{generation}) this implies that each subgroup  of the form $P^{(l)}\cap D$ is contained in $P_{j}$, for a suitable $j\leq t$. Thus $P^{(l)}=\langle P^{(l)}\cap P_{j} \mid 1 \leq j \leq t \rangle$, as required.    
\end{proof}

Combining Theorem \ref{generation1} with Lemma \ref{generation_and_product} we obtain a further refinement of  Theorem \ref{generation1}.

\begin{corollary}
\label{generation3}
$P=P_{1}P_{2}\cdots P_{t}$.
\end{corollary}
\begin{proof}
By Theorem  \ref{generation1} we have $P=\langle P_{1},\dots,P_{t}\rangle$. Since $P$ is nilpotent, in view of Lemma \ref{generation_and_product} it is sufficient to show that 
\begin{equation}
\label{eq}
\gamma_{i}(P)=\langle  \gamma_{i}(P)\cap P_{j} \mid 1\leq j \leq t\rangle,
\end{equation} 
for all $i\geq 1$. 

For $i=1$ the equality (\ref{eq}) is Theorem \ref{generation1}. 
Assume that $i\geq 2$. Set $N_{j}=\gamma_{i}(P)\cap P_{j}$ for $j=1,\ldots,t$ and $N=\langle N_j \mid 1\leq j \leq t \rangle$. 
By Lemma \ref{FG2} $[N_{j},P_{k}]$  can be generated by subgroups of the form $[N_{j},P_{k}]\cap C_{G}(A_{l})$, where $l=1,\ldots,s$ and  each of them is contained in some $N_{u}$ for suitable $u\leq t$. Indeed, on the one hand $[N_{j},P_{k}]\cap C_{G}(A_{l})$ is obviously contained in $\gamma_{i}(P)$. On the other hand it follows from  Proposition \ref{PAspecial}(1) that $[N_{j},P_{k}]\cap C_{G}(A_{l})$ is contained in some $A$-special subgroup $H$  of degree $d$. Hence 
$$
[N_{j},P_{k}]\cap C_{G}(A_{l})\leq\gamma_{i}(P)\cap (P\cap H)= \gamma_{i}(P)\cap P_{u}
$$
 for some $u\leq t$. So $[N_{j},P_{k}]\cap C_{G}(A_{l})$ is contained in some $N_{u}$ as desired.  This implies that $[N_{j},P_{k}]\leq N$ for all $j$ and $k$. Therefore $N$ is normal in $P$.  

We can now consider the quotient $P/N$ and observe that for $j=1,\ldots,t$ the image of the subgroup $\gamma_{i}(P)\cap P_{j}$ is trivial. Therefore $\gamma_i(P)\leq N$. Since the subgroup $N$ is obviously contained in $\gamma_{i}(P)$ we conclude that $N=\gamma_{i}(P)$ and we have (\ref{eq}). 
\end{proof} 

We will also require the following result that is a little stronger than Corollary \ref{generation3}.

\begin{corollary}
\label{generation4}
For all $l\geq 1$ the $l$th derived group $P^{(l)}$ is the product of the subgroups of the form $P^{(l)}\cap P_{j}$, where $j=1,\ldots,t$. 
\end{corollary}

\begin{proof}
Recall that by Lemma \ref{generation2} we have 
$$
P^{(l)}= \langle P^{(l)} \cap P_{j} \mid 1 \leq j \leq t \rangle
$$ for all  $l\geq 1$.  By  using the same argument as in the  proof  of Corollary \ref{generation3} the result follows.
\end{proof}
 

\section{Useful Lie-theoretic machinery}

\label{Lie_machinery}
Let $L$ be a Lie algebra over a field ${\mathfrak k}$. Let $k$ be a positive integer and let $x_1,x_2,\dots,x_k$ be elements of $L$. We define inductively 
$$[x_1]=x_1;\ [x_1,x_2,\dots,x_k]=[[x_1,x_2,\dots,x_{k-1}],x_k].$$
 An element $a\in L$ is called ad-nilpotent if there exists a positive integer $n$ such that 
 $$[x,\underset{n}{\underbrace{a,\ldots,a}}]=0 \quad \text{for all}\,  x\in L.$$ 
 If $n$ is the least integer with the above property then we say that $a$ is ad-nilpotent of index $n$. Let $X\subseteq L$ be any subset of $L$. By a commutator in elements of $X$ we mean any element of $L$ that can be obtained as a Lie product of elements of $X$ with some system of brackets. 
 
 Denote by $F$ the free Lie algebra over 
${\mathfrak k}$ on countably many free generators $x_1,x_2,\dots$. Let $f=f(x_1,x_2,\dots,x_n)$ be a non-zero element of $F$. The algebra $L$ is said to satisfy the identity $f\equiv 0$ if $f(a_1,a_2,\dots,a_n)=0$ for any $a_1,a_2,\dots,a_n\in L$. In this case we say that $L$ satisfies a polynomial identity, in short, is PI.  A deep result of Zelmanov \cite{Z0}, which has numerous important applications to group theory (in particular see \cite{OS} for examples where the theorem is used), says that if a Lie algebra $L$ is PI and is generated by finitely many elements all commutators in which are ad-nilpotent, then $L$ is nilpotent. From Zelmanov's result the following theorem can be deduced \cite{KS}. 

\begin{theorem}\label{liealgbnilp}
Let $L$ be a Lie algebra over a field ${\mathfrak k}$ generated by $a_1,a_2,\dots,a_m$.\  Assume that $L$ satisfies an identity $f\equiv 0$ and that each commutator in the generators $a_1,a_2,\dots,a_m$ is ad-nilpotent of index at most $n$.\ Then $L$ is nilpotent of $\{f,n,m,{\mathfrak k}\}$-bounded class.
\end{theorem}

The next theorem provides an important criterion for a Lie algebra to be PI. It was proved by Bakhturin  and Zaicev for soluble groups $A$ \cite{BZ} and later extended by Linchenko to the general case \cite{LI}.

\begin{theorem}
\label{LichBak}
 Assume that a finite group $A$ acts on a Lie algebra $L$ by automorphisms in such a manner that $C_L(A)$, the subalgebra formed by fixed elements, is PI. Assume further that the characteristic of the ground field of $L$ is either 0 or prime to the order of $A$. Then $L$ is PI.
\end{theorem}
We will need a corollary of the previous result. 

\begin{corollary}[\cite{Shu}]
\label{polynomialidentity}
Let $F$  be the free Lie algebra of countable rank over $\mathfrak k$. Denote by $F^{*}$ the set of non-zero elements of $F$. For any finite group $A$ there exists a mapping 
$$\phi: F^{*}\rightarrow F^{*}$$
such that if $L$ and $A$ are as in Theorem \ref{LichBak}, and if $C_{L}(A)$ satisfies an identity $f\equiv 0$, then $L$ satisfies the identity $\phi(f)\equiv 0$.
\end{corollary}

Now we turn to groups and  for the rest of this section $p$ will denote a fixed prime number. 
Let $G$ be any group. A series of subgroups 
\begin{equation*}
(*)\quad \quad  G=G_{1}\geq G_{2}\geq \cdots
\end{equation*}
 is called an $N_{p}$-series if $[G_{i},G_{j}]\leq G_{i+j}$ and $G_{i}^{p}\leq G_{pi}$ for all $i,j$. With any $N_{p}$-series $(*)$ of  $G$ one can associate a Lie algebra $L^{*}(G)=\oplus L^{*}_{i}$ over  the field with $p$ elements $\F_{p}$, where we view each  $L^{*}_{i}=G_{i}/G_{i+1}$ as a linear space over $\F_{p}$. If $x \in G$, let $i=i(x)$ be the largest integer such that $x \in G_i$. We denote by $x^*$ the element $xG_{i+1}$ of $L^{*}(G)$. The following lemma tells us something about the relationship between the group $G$ and the associated Lie algebra $L^{*}(G)$.

\begin{lemma}[Lazard, \cite{L}]
\label{Laz} 
For any $x\in G$ we have $(ad\, x^*)^p=ad\, (x^p)^*$. Consequently, if $x$ is of finite order $p^{t}$, then $x^*$ is ad-nilpotent of index at most $p^{t}$.
\end{lemma}

Let $w=w(x_1,x_2,\dots,x_n)$ be nontrivial group-word, i.e., a nontrivial element of the free group on free generators $x_1,x_2,\dots,x_n$.  We say that $G$ satisfies the identity $w\equiv 1$ if  $w(g_1,\dots,g_n)=1$ for any $g_1,g_2,\dots,g_n\in G$. The next proposition follows from the proof of  Theorem  1 in the paper of Wilson and Zelmanov \cite{WZ}. 

\begin{proposition}
\label{WilZel}
Let $G$ be a group satisfying an identity $w\equiv 1$. Then there exists a non-zero multilinear Lie polynomial $f$ over $\F_{p}$, depending only on $p$ and $w$, such that  for any  $N_{p}$-series  $(*)$ of $G$ the corresponding algebra  $L^{*}(G)$ satisfies the identity $f\equiv 0$.
\end{proposition}

In general a group  $G$ has many $N_{p}$-series; one of the most important is the so-called Jennings-Lazard-Zassenhaus series that can be defined as follows.   

Let $\gamma_j(G)$ denote the $j$th term of the lower central series of $G$.
Set $D_i=D_i(G)= \prod_{jp^{k}\geq i}\gamma_{j}(G)^{p^{k}}$. The subgroup $D_{i}$ is also known as  the $i$th-dimension subgroup of $G$ in characteristic $p$. These subgroups form an $N_{p}$-series of $G$ known as the  Jennings-Lazard-Zassenhaus series. Let $L_{i}=D_{i}/D_{i+1}$ and $L(G)=\oplus L_{i}$.  Then $L(G)$  is a Lie algebra over the field $\F_p$ (see \cite[Chapter 11]{GA} for more detail). The subalgebra of $L(G)$ generated by $L_{1}=D_1/D_2$ will be denoted by $L_p(G)$. The  next lemma is a ``finite'' version of Lazard's criterion for a pro-$p$ group to be $p$-adic analytic. The proof can be found in \cite{KS}.
\begin{lemma}
\label{PowerfulR}
Suppose that $P$ is a $d$-generator finite $p$-group such that the Lie algebra $L_{p}(P)$ is nilpotent of class $c$. Then $P$ has a powerful characteristic subgroup of $\{p,c,d\}$-bounded index.
\end{lemma}

Remind that powerful $p$-groups were introduced by Lubotzky and Mann in \cite{LM}.  A finite $p$-group $G$ is said to be powerful if and only if $[G,G]\leq G^{p}$ for $p \neq 2$ (or $[G,G]\leq G^{4}$ for $p=2$).  These groups have some nice properties. In particular we will use the following property: if $G$ is a powerful $p$-group  generated by elements of order $e=p^{k}$, then the exponent of $G$ is $e$. 

   Every subspace (or just an element) of $L(G)$ that is contained in $D_i/D_{i+1}$ for some $i$ will be called homogeneous. Given a  subgroup $H$ of  the group $G$, we denote by $L(G,H)$ the linear span in $L(G)$ of all homogeneous elements of the form $hD_{i+1}$, where $h\in D_{i}\cap H$.  Clearly, $L(G,H)$ is always a subalgebra of $L(G)$. Moreover, it is isomorphic with the Lie algebra associated with $H$ using the $N_{p}$-series of $H$ formed by $H_{i}=D_{i}\cap H$.  
We also set $L_{p}(G,H)=L_{p}(G)\cap L(G,H)$.  The proof of the following lemma can be found in \cite{GS}.

\begin{lemma}
\label{L(GH)}
Suppose that any Lie commutator in homogeneous elements $x_{1},\ldots,x_{r}$ of $L(G)$ is ad-nilpotent of index at most $t$.\  Let $K=\langle x_{1},\ldots,x_{r}  \rangle$ and assume that $K\leq L(G,H)$ for some subgroup  $H$ of $G$ satisfying a group identity $w\equiv 1$. Then there exists some $\{r,t,w,p\}$-bounded number $u$ such that:
$$[L(G),\underset{u}{\underbrace{K,\ldots,K}}]=0.$$  
\end{lemma}

Lemma \ref{FG1}(1) has important implications in the context of associated Lie algebras and their automorphisms. Let $G$ be a group with a coprime automorphism $a$. Obviously $a$ induces an automorphism of every quotient $D_i/D_{i+1}$. This action extends to the direct sum $\oplus D_i/D_{i+ 1}$. Thus, $a$ can be viewed as an automorphism of $L(G)$ (or of $L_p(G)$).  Set $C_i=D_i \cap C_G(a)$.\ Then Lemma \ref{FG1}(1) shows that 
\begin{equation}
\label{LandCand}
C_{L(G)}(a)=\oplus C_iD_{i+1}/D_{i + 1},
\end{equation}
and that 
\begin{equation}
\label{CLp}
C_{L_{p}(G)}(a)=L_{p}(G,C_{G}(a)).
\end{equation}
This implies that the properties of $C_{L(G)}(a)$ are very much related to those of $C_G(a)$. In particular, Proposition \ref{WilZel} shows that if $C_G(a)$ satisfies  a  certain identity, then $C_{L(G)}(a)$ is PI.

\section{Proof of the main result}
\label{derived subgroups case}

Our goal in this section is to prove that part (2) of Conjecture \ref{255} is  correct.  More precisely we have the following result.

\begin{theorem}\label{PR}
Let  m be a positive integer, $q$ a prime, and  $A$ an elementary abelian group of order $q^r$, with $r\geq2$. Suppose that $A$ acts as a coprime group of automorphisms on a finite group $G$.  If, for some integer $d$ such that $2^{d}\leq r-1$, the $d$th derived group of $C_G(a)$ has exponent dividing $m$ for any $a \in  A^{\#}$, then the $d$th derived group $G^{(d)}$ has $\{m,q,r\}$-bounded exponent.
\end{theorem}

First we will consider the particular case where $G$ is a powerful $p$-group.
 
\begin{lemma}
 \label{powerfulth}
 Theorem \ref{PR} is valid if $G$ is powerful.
 \end{lemma}
 
 \begin{proof}
It follows from \cite[Exercise 2.1]{GA} that $G^{(d)}$ is also powerful. Furthermore, by Proposition \ref{PAspecial}(3), $G^{(d)}$  is generated by $A$-special subgroups of $G$ of degree $d$.  Since $2^{d}\leq r-1$,  Proposition \ref{PAspecial}(4) shows that any $A$-special subgroup $H$ of $G$ of degree $d$ is contained in $C_{G}(B)^{(d)}$ for some nontrivial subgroup $B\leq A$ and so $H$  is also contained in $C_{G}(a)^{(d)}$ for some $a\in A^{\#}$.  This implies that $G^{(d)}$ is generated by elements of order dividing $m$, and so it follows from \cite[Lemma 2.5]{GA} that the exponent of $G^{(d)}$ divides $m$. 
 \end{proof}

We will now handle  the case of an arbitrary $p$-group. The Lie-theoretic techniques that we have described in Section \ref{Lie_machinery} will play a fundamental role in the subsequent arguments. 

\begin{lemma}
 \label{pgroupth}
 Theorem \ref{PR} is valid if $G$ is a $p$-group.
 \end{lemma}
\begin{proof} Assume that $G$ is a $p$-group. By Corollary \ref{generation3} we have
\begin{equation}
\label{producto}
G^{(d)}=G_{1}G_{2}\cdots G_{t},
\end{equation}
 where each $G_{j}$ is an $A$-special subgroup of $G$ of degree $d$.  It is clear that the number $t$ is $\{q,r\}$-bounded.  

 Let $x$ be any element of $G^{(d)}$. In view of  (\ref{producto}) we can write $x=x_{1}x_{2}\cdots x_{t}$, where each $x_{j}$ belongs to $G_{j}$. Since $2^{d}\leq r-1$, by Proposition \ref{PAspecial}(4) each $G_{j}$ is contained in  $C_{G}(B)^{(d)}$ for some subgroup $B\leq A$ such that $|A/B|\leq q^{2^{d}}$. Thus each $x_{j}$ is contained in some $C_{G}(a)^{(d)}$ for a suitable $a\in A^{\#}$.  
 
Let $Y$ be the subgroup of $G$ generated by the orbits $x_{j}^{A}$ for $j=1,\ldots,t$. Each orbit contains at most $q^{r-1}$ elements so it follows that $Y$ has at most $q^{r-1}t$ generators, each of order dividing $m$. Since $x\in Y$ and we wish to bound the order of $x$, it is enough to show that the exponent of $Y$ is $\{m,q,r\}$-bounded.  

Set $Y_{j}= G_{j}\cap Y$ for $j=1,\dots,t$ and note that  every $Y_{j}\leq C_{G}(a)^{(d)}$ for a suitable $a\in A^{\#}$. 
 Since $Y=\langle x_{1}^{A},\ldots, x_{t}^{A}\rangle$ and every $G_{j}$ is an $A$-invariant subgroup we have $Y=\langle Y_{1},\ldots,Y_{t}\rangle$.  By applying Lemma \ref{generation_and_product}  we see that  $Y=Y_{1}Y_{2}\cdots Y_{t}$. 

Let $L=L_{p}(Y)$ and let $V_{1},\ldots,V_{t}$ be the images of $Y_{1},\ldots,Y_{t}$ in $Y/\Phi(Y)$.  It follows that  the Lie algebra $L$ is generated by $V_{1},\ldots,V_{t}$.  

Let $W$ be a subspace of $L$. We say that $W$ is a \emph{special subspace} of weight $1$ of $L$ if and only if $W=V_{j}$ for some $j\leq t$ and say that $W$ is a special subspace of weight $\varphi\geq 2$ if $W=[W_{1},W_{2}]\cap C_{L}(A_{k})$, where $W_{1},W_{2}$ are some special subspaces of $L$ of weight $\varphi_{1}$ and $\varphi_{2}$ such that $\varphi_{1}+\varphi_{2}=\varphi$ and $A_{k}$ is some maximal subgroup of $A$ for a suitable $k$.

We wish to show that  every special subspace $W$ of $L$  corresponds to a subgroup of an $A$-special subgroup of $G$ of degree $d$. We argue by induction on the weight $\varphi$. If $\varphi=1$, then $W=V_{j}$  and so $W$ corresponds to $Y_{j}$ for some $j\leq t$.  Assume that $\varphi\geq 2$ and  write $W=[W_{1},W_{2}]\cap C_{L}(A_{k})$. By induction we know that $W_{1}, W_{2}$ correspond respectively to some $J_{1},J_{2}$ which are subgroups of some  $A$-special subgroups of  $G$ degree $d$. Note that $[W_{1},W_{2}]$ is contained in the image of $[J_{1},J_{2}]$. This implies that the special subspace $W$ corresponds to a subgroup of $[J_{1},J_{2}]\cap C_{G}(A_{k})$ which, by Proposition \ref{PAspecial}(1), is contained in some $A$-special subgroup of $G$ of degree $d$, as desired.  Moreover it follows from Proposition \ref{PAspecial}(4) that  every element   of  $W$ corresponds to some element of $C_{G}(a)^{(d)}$ for some $a\in A^{\#}$ and so, by Lemma \ref{Laz}, it is ad-nilpotent of index at most $m$.

From the previous argument we  deduce that  $L=\langle V_{1}, \ldots, V_{t}\rangle$ is generated by ad-nilpotent elements of index at most $m$ but we cannot claim that every Lie commutator in these generators is again  in some special subspace of $L$ and hence it is ad-nilpotent  of bounded index. To overcome this difficulty we  extend the ground field $\F_{p}$ by a primitive $q$th root of unity $\omega$ and put $\overline{L}=L\otimes\F_{p}[\omega]$. We view $\overline{L}$ as a Lie algebra over $\F_{p}[\omega]$ and it is natural to identify $L$ with the $\F_{p}$-subalgebra $L\otimes 1$ of $\overline{L}$.
In what follows we write $\overline{X}$ to denote $X\otimes \F_{p}[\omega]$ for some subspace $X$ of $L$. Note that if an element $x \in L$ is ad-nilpotent, then the ``same" element $x \otimes 1$ is also ad-nilpotent in $\overline{L}$. We will say that an element of $\overline{L}$ is homogeneous if it belongs to $\overline{S}$ for some homogeneous subspace $S$ of $L$. 

Let $W$ be a special subspace of $L$. We claim that
 
 \begin{equation}
\label{eigenvectorsadnilp}
\begin{aligned}
& \text{there exists an}\, \{m,q\}\text{-bounded number}\, u\, \text {such that every}\\
& \text{element}\, w\, \text{of}\, \overline{W}\, \text{is ad-nilpotent of index at most}\, u. \\ 	
\end{aligned}
\end{equation}
Since $w$ is a homogeneous element of $\overline{L}$ it can be written as 
$$
w=l_{0}\otimes1+l_{1}\otimes \omega+\cdots+l_{q-2}\otimes \omega^{q-2},
$$ 
for suitable homogeneous elements  $l_{0},\ldots,l_{q-2}$ of $W$.  The elements $l_{0},\ldots,l_{q-2}$ correspond to some $x_{0},\ldots,x_{q-2}$ of $Y$  that  belong to some $A$-special subgroup of degree $d$ and so in particular $x_{0},\ldots,x_{q-2}$ are   elements of $C_{G}(a)^{(d)}$ for some $a\in A^{\#}$. Set $H=\langle x_{0},\ldots,x_{q-2}\rangle$ and $K=\langle l_{0},\ldots, l_{q-2}\rangle$. Since $H$ has exponent $m$ and $K\leq L_{p}(Y, H)$, Lemma \ref{L(GH)} shows that there exists an $\{m,q\}$-bounded number $u$ such that 
\begin{equation}
\label{Kequ}
[L, \underset{u}{\underbrace{K,\ldots, K}}]=0
\end{equation}  
 Obviously (\ref{Kequ}) implies  that 
 \begin{equation}
 \label{Kbar}
 [\overline{L}, \underset{u}{\underbrace{\overline{K},\ldots, \overline{K}}}]=0.
 \end{equation} 
Since $w$ lies in $\overline{K}$,  (\ref{eigenvectorsadnilp}) follows.  

The group $A$ acts naturally on $\overline{L}$ and now the ground field is a splitting field for $A$. Since $Y$ can be generated by at most $q^{r-1}t$ elements, we can choose elements $v_{1},\ldots, v_{s}$ in $\overline{V_{1}}\cup \cdots \cup \overline{V_{t}}$ with $s\leq q^{r-1}t$  that generate the Lie algebra $\overline{L}$ and each of them is a common eigenvector for all transformations from $A$.

Now let $v$ be any Lie commutator  in $v_{1},\ldots,v_{s}$. We wish to show that $v$ belongs to some $\overline{W}$, where $W$ is a special subspace of $L$.  If $v$ has weight $1$ there is nothing to prove. Assume $v$ has weight at least 2. Write $v=[w_{1},w_{2}]$ for some $w_{1}\in \overline{W_{1}}$ and $w_{2}\in \overline{W_{2}}$, where $W_{1}, W_{2}$ are two special subspaces of $L$ of smaller weights. It is clear that  $v$ belongs to $[\overline{W_{1}},\overline{W_{2}}]$. Note that any commutator in common eigenvectors is again a common eigenvector. Therefore $v$ is a common eigenvector and it follows that there exists some maximal subgroup $A_{l}$ of $A$ such that $v\in C_{\overline{L}}(A_{l})$. Thus $v\in [\overline{W_{1}},\overline{W_{2}}]\cap C_{\overline{L}}(A_{l})$. Hence $v$ lies in  $\overline{W}$, where $W$ is the special subspace of $L$ of the form $[W_{1},W_{2}]\cap C_{L}(A_{l})$ and so by (\ref{eigenvectorsadnilp}) $v$ is ad-nilpotent of bounded index.
This proves that
\begin{multline}
 \label{commutatorsvi}
 \text{any commutator in}\,  v_{1},\ldots,v_{s}\, \text{is ad-nilpotent of index at most}\, u. 
  \end{multline}

 Remind that $C_{L}(a)=L_{p}(Y,C_{Y}(a))$. Proposition \ref{WilZel} shows that $C_{L}(a)$ satisfies a multilinear polynomial identity  of $\{m,q\}$-bounded degree. This also holds in $C_{\overline{L}}(a)=\overline{C_{L}(a)}$. Therefore Corollary \ref{polynomialidentity} implies that $\overline{L}$ satisfies a polynomial identity of $\{m,q\}$-bounded degree. Combining  this with (\ref{eigenvectorsadnilp}) and  (\ref{commutatorsvi})  we are now able  to apply Theorem \ref{liealgbnilp}. Thus $\overline{L}$ is nilpotent of $\{m,q,r\}$-bounded class and the same holds for $L$.   

Since $Y$ is a $p$-group and $L=L_{p}(Y)$ is nilpotent of bounded class,  it follows  from  Lemma \ref{PowerfulR}  that $Y$ has a characteristic powerful subgroup $K$ of $\{m,q,r\}$-bounded index. By Lemma \ref{powerfulth} $K^{(d)}$ has bounded exponent  and so we can pass to the quotient   $Y/K^{(d)}$  and assume that $Y$ is of $\{m,q,r\}$-bounded derived length.  We now recall that $Y=Y_{1}Y_{2}\ldots Y_{t}$ and  each $Y_{j}$ is contained is some $G_j$. From the results obtained in Section 4 also each derived group $Y^{(i)}$ is a product of subgroups of the form $Y^{(i)}\cap Y_j$. Thus every $Y^{(i)}$ can be generated by elements whose orders divide $m$. Since the derived length of $Y$ is bounded, we conclude that $Y$ has $\{m,q,r\}$-bounded exponent, as required.   
\end{proof}

Finally we are ready to complete the proof of Theorem \ref{PR}. 

\begin{proof}[Proof of Theorem \ref{PR}]
Note that it suffices to prove that there is a bound, depending only on $m,q$ and $r$, on the exponent of  a Sylow $p$-subgroup of $G^{(d)}$  for each prime $p$.  

Indeed, let $\pi(G^{(d)})$ be the set of prime divisors of  $|G^{(d)}|$. Choose $p\in \pi(G^{(d)})$. It follows from Lemma \ref{FG1}(2) that $G^{(d)}$ possesses an $A$-invariant Sylow $p$-subgroup, say $P$.  By  Corollary \ref{generation3},  $P=P_{1}P_{2}\cdots P_{t}$, where each $P_{j}$ is of the form $P\cap H$ for some $A$-special subgroup $H$ of $G$ of degree $d$.  Combining this fact with Proposition \ref{PAspecial}(4) we see that each $P_{j}$ is contained in $C_{G}(B)^{(d)}$ for a suitable subgroup $B$ of $A$ and thus  $P_{j}\leq C_{G}(a)^{(d)}$, for some $a\in A^{\#}$. Since the exponent of $C_{G}(a)^{(d)}$ divides $m$,  so does $p$.     
  
From Lemma \ref{pgroupth} we know that  $P^{(d)}$ has $\{m,q,r\}$-bounded  exponent.  Moreover  by Lemma \ref{generation2} the subgroup $P^{(d-1)}$ is generated by subgroups of the form $P^{(d-1)}\cap P_{j}$, for $j=1,\ldots,t$, so in particular  $P^{(d-1)}$ is generated by elements of order dividing $m$.   Since  $P^{(d)}=(P^{(d-1)})'$  has bounded exponent it is clear that also the exponent of $P^{(d-1)}$ is $\{m,q,r\}$-bounded.  Repeating the same argument several times we see that  all subgroups $P^{(d-2)},\ldots,P'$ and $P$ are generated by elements whose orders divide $m$ and so we conclude that $P$ has $\{m,q,r\}$-bounded exponent, as desired. This completes the proof. 
\end{proof}

\section{The other part of the conjecture}
\label{gamma case}

In this last section we will  deal with part (1) of Conjecture \ref{255}. The proof of that part is similar to that of part (2) but in fact it is easier. Therefore we will not give a detailed proof here but rather describe only steps where the proof of part (1) is somewhat different from that of part (2).

The definition of $A$-special subgroups of $G$ needs to be modified in the following way.

\begin{definition}
\label{gammaAspecial}
Let $A$ be an elementary abelian $q$-group acting on  a finite $q'$-group $G$.
Let $A_{1},\ldots,A_{s}$ be the subgroups of index $q$ in $A$ and $H$ a subgroup of $G$. 
\begin{itemize}
\item
We say that $H$ is a $\gamma$-$A$-special subgroup of $G$ of degree $1$ if and only if $H=C_{G}(A_{i})$ for suitable $i\leq s$.  

\item Suppose that $k\geq 2$ and the $\gamma$-$A$-special subgroups of $G$ of degree $k-1$ are defined. Then $H$ is a $\gamma$-$A$-special subgroup of $G$ of degree $k$ if and only if there exists a $\gamma$-$A$-special subgroup $J$ of $G$ of degree $k-1$ such that  $H=[J,C_{G}(A_{i})]\cap C_{G}(A_{j})$ for suitable $i,j\leq s$.  
\end{itemize}
 \end{definition}
 
The next proposition is similar to Proposition \ref{PAspecial}.

\begin{proposition}
\label{gammaPAspecial}
Let $A$ be an elementary abelian $q$-group of order $q^{r}$ with $r\geq 2$ acting on  a finite $q'$-group $G$
and $A_{1},\ldots,A_{s}$ the maximal subgroups of $A$. Let $k\geq 1$ be an integer. 

\begin{enumerate}
\item If $k\geq 2$, then every $\gamma$-$A$-special subgroup of $G$ of degree $k$ is contained in some $\gamma$-$A$-special subgroup of $G$ of degree $k-1$.

\item Let $R_{k}$ be the subgroup generated by all $\gamma$-$A$-special subgroups of $G$ of degree $k$. Then $R_{k}=\gamma_{k}{(G)}$.
\item If $k\leq r-1$ and $H$ is a $\gamma$-$A$-special subgroup of $G$ of degree $k$, then $H\leq \gamma_{k}(C_{G}(B))$ for some subgroup  $B\leq A$ such that $|A/B|\leq q^{k}$.

\item Suppose that $G=G'$ and let $N$ be an $A$-invariant subgroup such that $N=[N,G]$. Then for every $k\geq 1$ the subgroup $N$ is generated by subgroups of the form  $N\cap H$, where $H$ is some $\gamma$-$A$-special subgroup of $G$ of degree $k$.

\item Let $H$ be a $\gamma$-$A$-special subgroup of $G$. If $N$ is an $A$-invariant normal subgroup of $G$, then the image of $H$ in $G/N$ is a $\gamma$-$A$-special subgroup of $G/N$.  

 \end{enumerate}
  
\end{proposition}

The above properties of $\gamma$-$A$-special subgroups are essential in the proof  
of the following generation result, which is analogous  to Theorem \ref{generation1}.   
\begin{theorem}
\label{gamma_generation1}
Assume $r\geq 2$. Let $P$  be an $A$-invariant Sylow $p$-subgroup of $\gamma_{r-1}(G)$. Let $P_{1},\ldots,P_{t}$ be all the subgroups of the form $P\cap H$ where $H$ is some $\gamma$-$A$-special subgroup of $G$ of degree $r-1$. Then $P=\langle P_{1},\dots,P_{t}\rangle$.
\end{theorem} 

From this one can deduce

\begin{theorem}
\label{gamma_PR}
Let  $m$ be a positive integer, $q$ a prime and $A$ an elementary
abelian group of order $q^r$ with $r\ge 2$ acting 
on a finite $q'$-group $G$. If $\gamma_{r-1}(C_G(a))$ has exponent dividing
$m$ for any $a\in A^\#$, then $\gamma_{r-1}(G)$
has $\{m,q,r\}$-bounded exponent.
\end{theorem}

The above theorem shows that part (1) of Conjecture \ref{255} is correct. The proof of Theorem
\ref{gamma_PR} can be obtained in the same way as that of Theorem \ref{PR} with only obvious changes required. Thus, we omit the further details.


\end{document}